\documentclass[a4paper,12pt]{article}

\usepackage[english]{babel}

\usepackage[letterpaper,top=2cm,bottom=2cm,left=1.5cm,right=1.5cm,marginparwidth=1.75cm]{geometry}
\usepackage{graphicx}
\usepackage{mathtools}
\usepackage{amssymb}
\usepackage{tikz-cd}
\usepackage{authblk}
\usepackage[colorlinks=true, linkcolor=red, citecolor=blue, backref=page]{hyperref}
\usepackage{graphicx}
\usepackage{amsmath, amssymb, amsthm, mathtools, cancel, graphicx, wrapfig, forest, verbatim, amsfonts, bbm, caption, subcaption}

\renewcommand*{\backref}[1]{}
\renewcommand*{\backrefalt}[4]{%
  \ifcase #1 %
  \or
    ↑#2.%
  \else
    ↑#2.%
  \fi
}


\newtheorem{theorem}{Theorem}
\newtheorem{proposition}[theorem]{Proposition}
\newtheorem{corollary}[theorem]{Corollary}
\newtheorem{remark}[theorem]{Remark}
\newtheorem{definition}[theorem]{Definition}
\newtheorem{lemma}[theorem]{Lemma}
\newtheorem*{theorem-non}{Theorem}

\newcommand{\QQ}{\mathbb{Q}}
\newcommand{\ZZ}{\mathbb{Z}}
\newcommand{\CC}{\mathbb{C}}

\title{An archimedean approach to singular moduli on Shimura curves}
\date{}
\begin{document}
\author[M. Crabit Nicolau]{Mateo Crabit Nicolau\footnote{Sorbonne Université, Université Paris Cité, CNRS, IMJ-PRG, F-75005 Paris, France.\\ Gorlaeus Building, Einsteinweg 55, 2333 CC, Leiden. } }
\maketitle
\abstract{
We give a new proof of a recent generalization to Shimura curve of genus 0 of the work of Gross and Zagier in `On singular moduli'. This generalization  was conjectured by Giampietro and Darmon and proved by Daas by using $p$-adic $\Theta$-functions as an analogue of the $j$-invariant. Instead of working $p$-adically, we prove this result by evaluating Green's function at CM points on the Shimura curve. Our strategy is inspired by the analytic proof of Gross--Zagier. We put a special emphasis on both the similarities and the differences with the $p$-adic proof. 
}
{
  \hypersetup{linkcolor=black}
  \tableofcontents
}
\section{Introduction}
Let $j(\tau)$ denote Klein's $j$ invariant 
\[j(\tau)=q^{-1}+744+196884q+21493760q^2+...=q^{-1}+\sum_{n\geq 0}c_nq^n,\]
with $q=e^{2i\pi\tau}$ and $c_n\in\mathbb{Z}$. This holomorphic function defined on the upper half plane \[\mathbf{H}=\{\tau\in\mathbb{C},\textup{Im}(\tau)>0\}\]
is invariant under $\mathrm{SL}_2(\ZZ)$ acting by M\"{o}bius transformations on $\mathbf{H}$. A remarkable property is that when $\tau$ lies in an imaginary quadratic field $K$, the value $j(\tau)$ is an algebraic integer in an abelian extension of $K$.
In this setting, we call $j(\tau)$ a singular modulus.
\par
In \cite{GrossZagier84}, Gross and Zagier looked at the difference of two singular moduli 
 $j(\tau_1)-j(\tau_2)$ taken at CM points of coprime fundamental discriminants $D_1$ and $D_2$. This quantity is an algebraic integer and lives in the compositum of the Hilbert class fields of $\mathbb{Q}(\tau_1)$ and $\mathbb{Q}(\tau_2)$. The particularly smooth factorization appearing in a multitude of examples such as 
\begin{equation}
j\biggl(\frac{1+i\sqrt{163}}{2}\biggl)-1728=-2^6\cdot3^6\cdot7^2\cdot11^2\cdot19^2\cdot127^3\cdot163,
\end{equation}
or
\begin{equation}\label{vaj j 67,163}
j\biggl(\frac{1+i\sqrt{67}}{2}\biggl)- j\biggl(\frac{1+i\sqrt{163}}{2}\biggl)=-2^{15}\cdot3^7\cdot5^3\cdot7^2\cdot13\cdot139\cdot331,
\end{equation}
 prompted Gross and Zagier to study the arithmetic properties of these factorizations.
\par
We fix the following notation to state their Theorem. They will be used throughout this article. 
Let $D_1,D_2$ be two negative coprime fundamental discriminants and $D=D_1D_2$. Let $K_i=\mathbb{Q}(\sqrt{D_i})$ be an imaginary quadratic field with ring of integers $\mathcal{O}_i$ for $i=1,2$. Let $F=\mathbb{Q}(\sqrt{D})$ a real quadratic field with ring of integers $\mathcal{O}_F$, and $\chi$ the genus character of $F$ associated with the decomposition $D=D_1D_2$. We have the following diagram.
\begin{center}
\begin{tikzpicture}

    \node (Q1) at (0,0) {$\mathbb{Q}$};
    \node (Q2) at (1.3,1.3) {$K_2$};
    \node (Q3) at (0,2.6) {$L\coloneqq K_1K_2$};
    \node (Q4) at (-1.3,1.3) {$K_1$};
    \node(Q5) at  (0,1.3) {$F$};

    \draw (Q1)--(Q2);
    \draw (Q1)--(Q4);
    \draw (Q3)--(Q4);
    \draw (Q2)--(Q3);
    \draw (Q1)--(Q5);
    \draw (Q5)--(Q3) node [pos=0.5, right,outer sep=0.cm] {$\chi$};

    \end{tikzpicture}   
\end{center}
\begin{theorem}[Gross--Zagier]\label{GZ}
     Let $w_i$ be the order of the group of units $\mathcal{O}_{i}^\times$ and 
    \begin{equation}
   \mathcal{J}(D_1,D_2)\coloneqq \prod_{\substack{[\tau_1],[\tau_2]\\ \textup{disc}(\tau_i)=D_i}}(j(\tau_1)-j(\tau_2))^{\frac{4}{w_1w_2}}, 
   \end{equation}
where the product is taken over the $\mathrm{SL}_2(\mathbb{Z})$-orbits of points $\tau_1$,$\tau_2$ of discriminants $D_1$ and $D_2$ respectively. For a prime $l$ with $\bigl(\frac{D}{l}\bigl)\neq -1$, define 
\[\epsilon(l)\coloneqq \begin{cases}
      \bigl(\frac{D_1}{l}\bigl) & \text{if  } (l,D_1)=1,\\
      \bigl(\frac{D_2}{l}\bigl) & \text{if  } (l,D_2)=1.
      \end{cases}\]
If $n=\prod l_i^{\alpha_i}$ with $\bigl(\frac{D}{l_i}\bigl)\neq -1$ for all $i$, we define $\epsilon(n)=\prod \epsilon(l_i)^{\alpha_i}$. Then
\begin{equation}
\mathcal{J}(D_1,D_2)^2=\pm \prod_{\substack{x^2<D\\ x^2\equiv D\bmod{4}}}\mathbf{F}\biggl(\frac{D-x^2}{4}\biggl)
\end{equation}
where 
\begin{equation}
\mathbf{F}(m)=\prod_{\substack{nn'=m\\ n,n'>0}}n^{\epsilon(n')}.
\end{equation}
\end{theorem}
 The quantity $\mathbf{F}(m)$ appearing in the Theorem is always a prime power. More precisely, $\mathbf{F}(m)\neq 1$ if and only if  there is only one prime $\ell\mid m$ such that $\epsilon(\ell)=-1$ and $v_\ell(m)$ is odd.
 \\
 \par
We state a generalization to Shimura curves of the work of \cite{GrossZagier84}. This generalization was conjectured by Giampietro and Darmon in \cite{GiampetroDarmon} and proved by Daas in \cite{Daasarticle}. Our work will rely on some results of Daas, we will refer to their PhD thesis \cite{Daas} where more details are given.

\par Let $B_N/\mathbb{Q}$ be an indefinite quaternion algebra with discriminant $N>0$. Let  $R$ be a maximal order of $B_N$. Because $B_N$ is indefinite, $R$ is unique up to conjugation. Let $R^1$ be the set of elements of norm $1$ in $R$. By choosing an isomorphism  ~$B_N\otimes\mathbb{R} \simeq M_2(\mathbb{R})$, $R^1$ acts by Möbius transformation on the upper half plane.
We define the Shimura curve 
\[X_N(\mathbb{C})\coloneqq R^1\backslash\mathbf{H},\]
which has the structure of an algebraic curve defined over $\mathbb{Q}$. It is of genus $0$ if and only if $N\in \{1,6,10,22 \}$. 
Fix $N\in\{6,10,22 \}$ and write $N=pq$. Because $X_N$ is of genus 0, we can choose an isomorphism  $j_N : X_N \stackrel{\sim}{\longrightarrow} \mathbb{P}^1$ defined over $\QQ$.  Unlike the case $N=1$, $X_N$ is compact so there are no cusps. One cannot use Fourier series to make a natural choice of $j_N$ like Klein's $j$-function.
\par
We define CM-points on $X_N$ as follows. Let $K$ be an imaginary quadratic field such that $p$ and $q$ are inert in $K$. In this case, we can find an embedding $\alpha:\mathcal{O}_K\rightarrow R$, giving an action of $\mathcal{O}_K$ on  $\mathbf{H}$ by Möbius transformation for which there is a unique fixed point $\tau\in\mathbf{H}$. The image $P\in X_N(\CC)$ of $\tau$ is called the CM-point associated to the embedding $\alpha$.
\par
As explained in the first pages of \cite{Shimura}, for any CM-points $P\in X_N(\CC)$, the value $j_N(P)$ is defined over the Hilbert class field $H$ of $K$ by the Shimura reciprocity law. We fix now an embedding $\overline{\QQ}\hookrightarrow\CC.$
\par
 We suppose that $p$ and $q$ are inert in $K_1$ and $K_2$. Our setting is now quite close to the case $N=1$ of Gross and Zagier. After extensive numerical experimentation, Giampietro and Darmon \cite{GiampetroDarmon} gave a precise conjecture of the formula for the factorization of the norm of a cross ratio of such $j_N(\tau)$. More precisely, let ${\pm a,\pm b}$ be the four square roots of $D$ modulo $2N$ and define for $x\in\mathbb{Z}$ such that $x^2\equiv D\,[4N]$,
\begin{equation}
\delta(x)\coloneqq \begin{cases}
      +1 & \text{if  } x\equiv \pm a \,[2N],\\
      -1 & \text{if  } x\equiv \pm b\, [2N].
      \end{cases}
\end{equation}
Then, with $w_p$ the Atkin--Lehner involution at $p$ on $X_N$ (see Equation (\ref{atkin lehner def})),  define $P'\coloneqq w_p(\text{Frob}_p(P))$ where $\text{Frob}_p$ is the Frobenius at $p$ in the CM field $H$ for a now fixed embedding $H\hookrightarrow \QQ_{p^2}$. The quantity we are interested in is 
\begin{equation}
\mathcal{J}_N(D_1,D_2)\coloneqq \textup{Nm}_\mathbb{Q}^{H_1H_2}\biggl(\frac{j_N(P_1)-j_N(P_2)}{j_N(P_1')-j_N(P_2)}\cdot
\frac{j_N(P_1')-j_N(P_2')}{j_N(P_1)-j_N(P_2')}\biggl)^{\frac{ 2}{w_1w_2}}.
\end{equation} 
We can now state the conjecture of Giampietro and Darmon proved by Daas. Note the striking resemblance with Theorem$\,$\ref{GZ}. 
\begin{theorem}[Daas]\label{padicgz}
    Let $N=pq\in\{6,10,22\}$, let $\alpha_i$ be an embedding of $\mathcal{O}_i$ and $P_i$ the CM point attached to $\alpha_i$. Then, with the same $\mathbf{F}$ as Theorem \ref{GZ},
    \begin{equation}
\mathcal{J}_N(D_1,D_2)^{\pm1}=\pm\prod_{\substack{x^2<D\\x^2\equiv D \bmod{4N}}
    }\mathbf{F}\biggl(\frac{D-x^2}{4N}\biggl)^{\delta(x)}. \end{equation}
\end{theorem}
Taking $N=6$, $D_1=-67$ and $D_2=-163$, Giampietro and Darmon computed
\begin{equation}
\textup{Nm}_\mathbb{Q}^{H_1H_2}\biggl(\frac{j_N(P_1)-j_N(P_2)}{j_N(P_1')-j_N(P_2)}\cdot
\frac{j_N(P_1')-j_N(P_2')}{j_N(P_1)-j_N(P_2')}\biggl)^4=\biggl(\frac{79^2\cdot101^2\cdot233^2}{5^2\cdot13^4\cdot31^2\cdot67^2}\biggl)^2.
\end{equation}
In \cite{Errthum_2011}, the values of the norm of $j_N(P_1)$ have been computed when $N=6,10$ using Borcherd's lifts. The reader can refer to \cite{GiampetroDarmon} \S 6 for an insightful comparison between the two methods. More computations for a particular choice of $j_N$  can be found in \cite{Voight2009}. 
\par
We now rewrite the quantity appearing in Theorem \ref{padicgz} in a more convenient way for our setting.
The choice of $P'$ is crucial for the computation in \cite{GiampetroDarmon} and the proof in \cite{Daas} as it allows to relate this cross ratio to the value of some $p$-adic $\Theta$-functions using the Čerednik-Drinfeld isomorphism. This choice is no longer essential in our case, as we will not work $p$-adically. Instead, define the following cross-ratio:
\begin{equation}\label{def J_n}
J_N(P_1,w_q(P_1),P_2,w_p(P_2))\coloneqq
\frac{j_N(P_1)-j_N(P_2)}{j_N(w_q(P_1))-j_N(P_2)}\cdot
\frac{j_N(w_q(P_1))-j_N(w_p(P_2))}{j_N(P_1)-j_N(w_p(P_2))}.
\end{equation}
The class group $\textup{Pic}(K_i)$ acts on the set of embeddings of $\mathcal{O}_i$ (see Equation (\ref{action of Pic})). By extension, it also acts on their fixed point.
Let $G=\textup{Pic}(K_1)\times \textup{Pic}(K_2)$, we show in Proposition \ref{rewriting J_n} that
\begin{equation}\label{equation def J_N}
\mathcal{J}_N(D_1,D_2)
=\Biggl(\prod_{\substack{(\sigma_1,\sigma_2)\in G}}
\biggl|J_N\Bigl(P_1^{\sigma_1},w_q(P_1)^{\sigma_1},P_2^{\sigma_2},w_p(P_2)^{\sigma_2}\Bigl)\biggl|^2\Biggl)^{\frac{4}{w_1w_2}}.
\end{equation}
We use this new formula for $\mathcal{J}_N(D_1,D_2)$ to prove Theorem \ref{padicgz}.
Our proof will follow the same steps as the original analytical proof of Theorem \ref{GZ}. First we will take a $s$-family of Hilbert Eisenstein series attached to the real quadratic field $F$ that will vanish at $s=0$. The diagonal restriction of the derivative in $s=0$ (see Definition~\ref{def Eis}) is a non-holomorphic cusp form of weight $2$ for $\Gamma_0(N)$. We then apply the holomorphic projection operator to compute the Fourier coefficients of a modular form in $S_2(\Gamma_0(N))$ (Theorem \ref{thm mod}) which must be 0 when $N=6,10,22$. The first Fourier coefficient $a_1$ is given explicitly as a difference of two terms $a_1=b_1-c_1=0.$ It is straightforward to verify that $b_1$ is the logarithm of the right-hand side of the equality in Theorem \ref{padicgz}. In Section \ref{section Quaternion} we prove that $c_1$ can be expressed as a sum of evaluations of Green's functions on $B_{N}$ at CM points. Finally, since $\tau_1,\tau_2\rightarrow\log(\mid j_{N}(\tau_1)-j_{N}(\tau_2)\mid ^2)$ is also a Green's function on the Shimura curve, we are able to conclude. 
\\
\\
\textbf{Acknowledgments} I would like to thank Jan Vonk for introducing me this problem, and for his guidance and insight during this work. I would like to thank Pierre Charollois for his help and his detailed feedback. Finally, thanks goes to Mike Daas for his interest in my work and his kind answers to my questions, and to Pierre Morain for his support. This work is part of an on-going PhD work and was supported by ERC Starting Grant "GAGARIN" 101076941. 
\section{Rewriting the equality}
Following the steps of the original analytical proof of Theorem \ref{GZ} of Gross--Zagier \cite{GrossZagier84}, we construct an explicit  cusp form $f$ of weight $2$ for $\Gamma_0(N)$ whose Fourier coefficients are related to Equation (\ref{equation def J_N}). To do so we will apply the holomorphic projection to a non-holomorphic cusp form of weight $2$ obtained from a real analytic family of Hilbert Eisenstein series. This will allow us to compute the Fourier coefficients of $f$. 
 When  $N=6,10,22$, this cusp form must vanish (see Corollary \ref{cor}). The goal of Section \ref{section Quaternion} will be to prove that the vanishing of the first Fourier coefficient of $f$ is equivalent to Theorem \ref{padicgz}.
\subsection{Eisenstein Series}
In this subsection we construct a non-holomorphic cusp form to which we will apply the holomorphic projection.
Recall the diagram given before Theorem \ref{GZ}. We also fix $N=pq$ with $p$ and $q$ two different primes inert in $K_1$ and $K_2$. They both split in $F$.
\par
 Following \cite{DDP} Section 2.2, let $\textup{Pic}(F)$ be the class group of $F$ and $\lambda\in \textup{Pic}(F)$ be an ideal class. We take $t_\lambda\subset \mathcal{O}_F$ an ideal in this class. We define for $\text{Re}(s)>1/2$ and $z,z'\in\mathbf{H}$,
\[E_{s,\lambda}(z,z')\coloneqq \sum_{[\mathfrak{a}]\in \textup{Pic}(F)}\chi(\mathfrak{a})N(\mathfrak{a})^{1+2s} E_{s,\lambda}^\mathfrak{a}, \]
with 
\[ E_{s,\lambda}^\mathfrak{a}\coloneqq \frac{1}{C_U}\sum_{\substack{m\in \mathfrak{a} \\ n\in \mathfrak{a} t_\lambda^{-1} \delta_F^{-1} \\ (m,n) \mod U \\ (m,n)\neq (0,0)}} \frac{(yy')^{s}}{(mz+n)(m'z'+n') \mid (mz+n)(m'z'+n')\mid^{2s}},\]
where $\delta_F$ is the different of $F$, $U\coloneqq \{ u\in \mathcal{O}_F^\times, N(u)=1,\, u \equiv 1 \mod t_\lambda \}$ which acts by diagonal multiplication on the two factors $(m,n)$ and $C_U=[\mathcal{O}_F^\times:U]$. The Hilbert Eisenstein series $E_{s,\lambda}(z,z')$ is a non holomorphic Hilbert modular form of weight $1$ for 
\[\Gamma_\lambda\coloneqq \biggl\{\begin{pmatrix} a & c\\ b& d\end{pmatrix}, \, a,d\in \mathcal{O}_F,\,\, b\in t_\lambda^{-1}\delta_F ^{-1},\,\, c\in t_\lambda \delta_F, \,\, ad-bc\in \mathcal{O}_F^\times \biggl\}.\]
The analytical proof in \cite{GrossZagier84} relies on the particular case $t_\lambda=\delta_F^{-1}$.
Each of these Eisenstein series has a holomorphic continuation in $s$ and when $\chi(t_\lambda)=-1$, we have $E_{0,\lambda}(z,z')=0$. It is then tempting to take the first derivative at $s=0$ and apply the holomorphic projection. However, the quotient in Theorem \ref{padicgz} leads us to take a difference of these Eisenstein series instead of taking a single one. We define now the non holomorphic modular form of weight 2 which we will apply the holomorphic projection to. 
\begin{definition}\label{def Eis}
    Let $p=\mathbf{p}_1\mathbf{p}_2$ and $q=\mathbf{q}_1\mathbf{q}_2$ the decomposition of $p$ and $q$ in $F$. We define 
\begin{equation}
 E_N(z)\coloneqq  \frac{D^{1/2}}{8\pi^2N}\sum_{i,j=1,2}(-1)^{i+j}\frac{\partial}{\partial s}E_{s,\delta_F^{-1}\mathbf{p}_i\mathbf{q}_j}(z,z)\biggl|_{s=0}
\end{equation}
\end{definition}
     Note that we took $t_\lambda=\delta_F^{-1}\mathbf{p}_i\mathbf{q}_j$ such that $\chi(t_\lambda)=-1.$
    Our definition of $E_N$ depends on the choice of $\mathbf{p}_1$ and $\mathbf{q}_1$, but the reader should keep in mind that we are trying to prove that the holomorphic projection of $E_N(z)$ is $0$, hence $E_N(z)$ needs only to be well defined up to sign.
\subsection{Holomorphic projection}
It is easily checked that each $E_{s,\delta_F^{-1}\mathbf{p}_i\mathbf{q}_j}$ is a particular case of the Eisenstein series appearing in \cite{GrossKohnenZagier87} \S III,1. Applying the holomorphic projection lemma of \cite{GrossKohnenZagier87} to our case, we can state the following proposition. First, define for an ideal $I$ of $\mathcal{O}_F$
\begin{equation}
    \sigma_{0,\chi}(I)=\sum_{\substack{J\mid I\\ J\subset \mathcal{O}_F}}\chi(J)\quad\textup{and } \quad \sigma'_{\chi}(I)=\sum_{\substack{J\mid I\\ J\subset \mathcal{O}_F}}\chi(J)\log(\textup{Nm}^F_{\QQ}(J)).
\end{equation}
\begin{proposition}[Holomorphic projection]\label{thm mod}
Let $\langle\,\cdot \, , \cdot \, \rangle_{Pet}$ be the Petersson inner product of weight $2$ for $ \Gamma_0(N)$. There  exists a unique modular form $f=\sum_{m>0}a_mq^m\in S_2(\Gamma_0(N))$ called the holomorphic projection of $E_N$ on $S_2(\Gamma_0(N))$ such that for all $g\in S_2(\Gamma_0(N))$, 
 \[\langle f,g\rangle_{Pet}=\langle E_N,g\rangle_{Pet}.\]
 Furthermore, when $m$ is prime to $N$, the Fourier coefficient $a_m$ is given by
 \[
        a_m=b_m-c_m,
    \]
 where
    \begin{equation}
b_m=\sum_{i,j=1,2}(-1)^{i+j}\sum_{\substack{\nu\in \delta_F^{-1}\mathbf{p}_i\mathbf{q}_j\\ \nu\gg0\\ \textup{Tr}_{F/\mathbb{Q}}(\nu)=m}}\sigma'_{\chi}((\delta_F^{-1}\mathbf{p}_i\mathbf{q}_j)^{-1}(\nu)),
\end{equation}
while
\begin{equation}
c_m=2\lim_{s\rightarrow 1} \sum_{i,j=1,2}(-1)^{i+j}\sum_{\substack{\nu\in \delta_F^{-1}\mathbf{p}_i\mathbf{q}_j\\ \nu'>0>\nu\\ \textup{Tr}_{F/\mathbb{Q}}(\nu)=m}}\sigma_{0,\chi}((\delta_F^{-1}\mathbf{p}_i\mathbf{q}_j)^{-1}(\nu))Q_{s-1}\Bigl(1-\frac{2\nu}{m}\Bigl),
\end{equation}
with $Q_{s-1}(t)$ the Legendre function of the second kind (See Equation (\ref{def Legendre function})). 

\end{proposition}
\begin{proof}
   For each $i,j=1,2$ define
\[f_{i,j} \coloneqq\pi_{hol}\Bigl(\frac{D^{1/2}}{8\pi^2N}\frac{\partial}{\partial s}E_{s,\delta_F^{-1}\mathbf{p}_i\mathbf{q}_j}(z,z)\bigl|_{s=0}\Bigl).
   \] 
   When $m$ is prime to $N$, \S III.2 Proposition 2 in \cite{GrossKohnenZagier87} gives a formula for the Fourier coefficient $a_{m,i,j}$ of $f_{i,j}$.
Take $r_{i,j}\in \mathbb{Z}$ such that 
\[ \mathbf{p}_i\mathbf{q}_j=N\mathbb{Z}+\frac{r_{i,j}+\sqrt{D}}{2}\mathbb{Z}.\]
Using the bijection
\begin{equation}
\begin{split}
    mr_{i,j}+2N\ZZ &\xrightarrow{\sim}\{\nu\in\delta_F^{-1}\mathbf{p}_i\mathbf{q}_j,\,\, \textup{Tr}_{F/\mathbb{Q}}(\nu)=m\}
    \\
    n&\rightarrow\frac{n+m\sqrt{D}}{2\sqrt{D}}
    \end{split}
\end{equation}
 we can rewrite $a_{m,i,j}$ as
    \begin{equation}
    \begin{split}
        a_{m,i,j}=&\sum_{\substack{\nu\in \delta_F^{-1}\mathbf{p}_i\mathbf{q}_j\\ \nu\gg0\\ \textup{Tr}_{F/\mathbb{Q}}(\nu)=m}}\sigma'_{\chi}((\delta_F^{-1}\mathbf{p}_i\mathbf{q}_j)^{-1}(\nu))
        \\&-\lim_{s\rightarrow 1} \sum_{\substack{\nu\in \delta_F^{-1}\mathbf{p}_i\mathbf{q}_j\\ \textup{Nm}_{F/\mathbb{Q}}(\nu)<0\\ \textup{Tr}_{F/\mathbb{Q}}(\nu)=m}}\sigma_{0,\chi}((\delta_F^{-1}\mathbf{p}_i\mathbf{q}_j)^{-1}(\nu))Q_{s-1}\Bigl(\Bigl|1-\frac{2\nu}{m}\Bigl|\Bigl) +\frac{\lambda_{D_1,D_2,N,m}}{s-1}
        \\&+C_{D_1,D_2,N,m}.
    \end{split}
    \end{equation}
 Where $\lambda_{D_1,D_2,N,m}\in\mathbb{Q}$ and $C_{D_1,D_2,N,m}\in\mathbb{C}$ are independent of the choice of $\mathbf{p}_i$ and $\mathbf{q}_j$
We end the proof by changing $\nu$ to $\nu'$ when $\nu>0>\nu'$ and taking our signed sum on $i,j=1,2$.
\end{proof}
The simplification of the constant $C_{D_1,D_2,N,m}$ in the proof reflects the fact that $E_{N}$ is a cuspidal (non holomorphic) modular form, which makes the computation of the holomorphic projection much easier than the more general case of \cite{GrossKohnenZagier87}.

\begin{remark}
    The analytic proof in \cite{Daas} follows a similar strategy, working at a prime $p \mid N$. A suitable $p$-adic family $E_{s}^{(p)}$ of Hilbert cusp forms on $F$ is computed using Galois deformation theory, and the ordinary projection is applied to the diagonal restriction of its first derivative, to obtain:
    \begin{equation}\label{p adic Eis Daas}
    e^{\textup{ord}}\biggl(\frac{\partial}{\partial s}E_{s}^{(p)}(z,z)\biggl) \ \in \ S_2(\Gamma_0(N)). 
      \end{equation}
Notice that both constructions are based on the same Hilbert modular form given by Proposition 2.1 of \cite{DDP}.
\end{remark}
\begin{corollary} \label{cor}
    When $N\in\{6,10,22\}$, the form $f$ from Proposition \ref{thm mod} is 0 and in particular:
    \begin{equation}\label{b1=c1}
        b_1=c_1,
    \end{equation}
    with 
    \begin{equation}
    b_1=\sum_{i,j=1,2}(-1)^{i+j}\sum_{\substack{\nu\in \delta_F^{-1}\mathbf{p}_i\mathbf{q}_j\\ \nu\gg0\\ \textup{Tr}_{F/\mathbb{Q}}(\nu)=1}}\sigma'_{\chi}((\delta_F^{-1}\mathbf{p}_i\mathbf{q}_j)^{-1}(\nu)),
     \end{equation}
    and 
    \begin{equation}
    c_1=2\lim_{s\rightarrow1} \sum_{i,j=1,2}(-1)^{i+j}\sum_{\substack{\nu\in \delta_F^{-1}\mathbf{p}_i\mathbf{q}_j\\ \nu'>0>\nu\\ \textup{Tr}_{F/\mathbb{Q}}(\nu)=1}}\sigma_{0,\chi}((\delta_F^{-1}\mathbf{p}_i\mathbf{q}_j)^{-1}(\nu))Q_{s-1}(1-2\nu).
     \end{equation}
\end{corollary}
\begin{proof}
    If $N=6,10$, then $S_2(\Gamma_0(N))=\{0\}$ so $f=0$. When $N=22$, the space is non-zero. However, let 
    \[W_{N}=\begin{pmatrix} 
    0 & -1 \\ N &0 
    \end{pmatrix}\]  
    be the matrix of the Atkin-Lehner involution. From \cite{GrossKohnenZagier87} Section $3.1$, we know that $f$ is invariant by $W_N$. The only form in $S_2(\Gamma_0(22))$ satisfying this condition is 0, hence $f=0$.
\end{proof}
\begin{remark}
    We will show that the equality $b_1=c_1$ is equivalent to our main theorem. In \cite{Daas}, the ordinary projection of the Eisenstein series (\ref{p adic Eis Daas}) will also be $0$ for $N=6,10,22$ and the vanishing of the first Fourier coefficient will be one of the main ingredients of the proof.
\end{remark}
Our goal now will be to show that this equality is equivalent to Theorem \ref{padicgz}. First, let's take a look at the first term $b_1$. Let $\pm a,\pm b$ the four square roots of $D \mod 2N$ when $N=6,10,22$. We choose $a$ such that $\mathbf{p}_1\mathbf{q}_1=N\mathbb{Z}+\frac{a+\sqrt{D}}{2}\mathbb{Z}$, we have the following proposition.
\begin{proposition}\label{prop b1} With the same notation as before:
    \[b_1=-\log\Biggl(\prod_{\substack{x^2<D\\x^2\equiv D \bmod{4N}}}\mathbf{F}\biggl(\frac{D-x^2}{4N}\biggl)^{\delta(x)}\Biggl).\]
\end{proposition}
\begin{proof}
This is a consequence of \cite{Daas} Proposition 1.6.1 and our choice of $\delta$.
\end{proof}
Now, looking at \[c_1=2\lim_{s\rightarrow 1} \sum_{i,j=1,2}(-1)^{i+j}\sum_{\substack{\nu\in \delta_F^{-1}\mathbf{p}_i\mathbf{q}_j\\ \nu'>0>\nu\\ \textup{Tr}_{F/\mathbb{Q}}(\nu)=1}}\sigma_{0,\chi}((\delta_F^{-1}\mathbf{p}_i\mathbf{q}_j)^{-1}(\nu))Q_{s-1}(1-2\nu),\]
 our goal for the next part is, in the way of \cite{GrossZagier84}, to relate our cross-ratio of $j_N$ invariant to a certain Green's function and show that its value at CM points is exactly the quantity that appears here. To do so, we introduce some linear structures on the quaternion algebra $B_N$. 
\begin{remark} When $\textup{Tr}_{F/\mathbb{Q}}(\nu)=1$ and $\nu>0>\nu'$, we have
    \[Q_0(1-2\nu')=\log\Bigl(1-\frac{1}{\nu'}\Bigl)=\log\Bigl(-\frac{\nu}{\nu'}\Bigl).\] Now, without any concern of convergence, this last expression for $c_1$ could be seen as an archimedean equivalent of the logarithm of the $p$-adic $\Theta$-function appearing in \cite{Daas} Proposition 6.1.4. 
   By switching the limit and the sum, we have the following heuristic for $c_1$ as a logarithm of an infinite product
    \begin{equation}\label{c1 infinit prod}
        \text{``}c_1=2\sum_{i,j=1,2}(-1)^{i+j}\log\biggl(\prod_{\substack{\nu\in \delta_F^{-1}\mathbf{p}_i\mathbf{q}_j\\ \nu'>0>\nu\\ \textup{Tr}_{F/\mathbb{Q}}(\nu)=1}}\Bigl(\frac{-\nu}{\nu'}\Bigl)^{\sigma_{0,\chi}((\delta_F^{-1}\mathbf{p}_i\mathbf{q}_j)^{-1}(\nu))}\biggl)\text{''}.
    \end{equation}
   Daas proves that the $p$-adic $\Theta$-function can be evaluated as
    \begin{equation}\label{eq theta fun}
        \Theta(D_1,D_2)^{\frac{\pm2}{w_1w_2}}=\lim_{n\rightarrow\infty}\prod_{\substack{\nu\in \delta_F^{-1}\mathbf{q}_1\\\nu\gg0\\ \textup{Tr}_{F/\mathbb{Q}}(\nu)=p^{2n}}}\Bigl(\frac{\nu}{\nu'}\Bigl)^{\sigma_{0,\chi}((\delta_F^{-1}\mathbf{q}_1)^{-1}(\nu))}.
    \end{equation}
Ignoring issues of convergence, we see a clear analogy between the infinite products (\ref{c1 infinit prod}) and (\ref{eq theta fun}) appearing in the archimedean and $p$-adic proof.
\end{remark}
\section{Quaternion algebra}\label{section Quaternion}
Let $B_{N}$ be the indefinite quaternion algebra ramified at $p$ and $q$. Let $R$ be a maximal order in $B_{N}$. Because $B_{N}$ is indefinite, $R$ is unique up to conjugation. The primes $p$ and $q$ are inert in $K_1$ and $K_2$, hence we have two embeddings $\alpha_1 : \mathcal{O}_{K_1}\xhookrightarrow{}R$, $\alpha_2:\mathcal{O}_{K_2}\xhookrightarrow{}R$. We then extend $\alpha_i:K_i\xhookrightarrow{}B_{N}$. Until the end, we choose one isomorphism $B_{N}\otimes\mathbb{R}\simeq \text{M}_2(\mathbb{R})$.
\subsection{Green's function}\label{subsection green function}
In this subsection, we introduce Green’s functions on the Shimura curve $X_N$. In particular, we show in Proposition \ref{Greenf} that when $N=6,10,22$, the function
\[(\tau_1,\tau_2)\rightarrow\log(|j_N(\tau_1)-j_N(\tau_2)|^2)
\]
can be interpreted as a Green's function (by abuse of notation, we define $~j_N(\tau)\coloneqq j_N(P)$ where $P\in X_N(\CC)$ is the image of $\tau\in \mathbf{H}$ by the projection $\mathbf{H}\rightarrow X_N(\CC)$). Later on, this will give us a way to compute the cross-ratio appearing in Theorem \ref{padicgz} explicitly.
\par
We will follow \cite{GrossGreenfunction} \S9. Let $\tau_1,\tau_2\in\mathbf{H}$ not equivalent under $R^1$, we define for $\textup{Re}(s)>1$,
\begin{equation}
g_s(\tau_1,\tau_2)\coloneqq -2Q_{s-1}(\text{cosh}(d(\tau_1,\tau_2)))=-2Q_{s-1}\Biggl(1+\frac{\mid \tau_1-{\tau_2}\mid^2}{2\text{\normalfont{Im}}(\tau_1)\text{\normalfont{Im}}(\tau_2)}\Biggl)
\end{equation}
and 
\begin{equation}\label{eq green function}
G_{N,s}(\tau_1,\tau_2)\coloneqq \sum_{\gamma\in R^1/\{\pm1\}}g_s(\tau_1,\gamma\tau_2),
\end{equation} 
where $R^1$ is the set of elements of $R$ of norm $1$. Recall that \begin{equation}\label{def Legendre function}
Q_{s-1}(t)=\int_{0}^\infty(t+\sqrt{t^2-1}\cosh(v))^{-s}dv
\end{equation}
is the Legendre function of the second kind. The series $G_{N,s}$ is convergent when Re$(s)>1$ and defines a symmetric bi-$R^1$-invariant function with a logarithmic pole on the diagonal which satisfies $\Delta G_{N,s}(\tau_1, \tau_2) = s(s - 1)G_{N,s}(\tau_1, \tau_2)$ for
fixed $\tau_2$. It has a simple pole at $s=1$ with residue $\text{Res}_{G_N}$ independent of $\tau_1,\tau_2$.
We can now define the Green's function:
\begin{equation}
G_N(\tau_1,\tau_2)\coloneqq \lim_{s\rightarrow 1}\Biggl(G_{N,s}(\tau_1,\tau_2)-\frac{\text{Res}_{G_N}}{s-1}\Biggl).
\end{equation}
Let $N=6,10,22$, as explained in the introduction, we have an isomorphism $j_{N}:X_N(\CC)\rightarrow \mathbb{P}^1(\mathbb{C})$. The following proposition relates the values of $J_N$ defined in Equation (\ref{def J_n}) and $G_N$.
\begin{proposition}\label{Greenf}
    Let $\tau_i\in \mathbf{H}$ corresponding to a CM-point associated to $K_i$ for $i=1,2$. Then
    \begin{equation}
    \begin{split}
        \log(\mid J_N(\tau_1,w_q(\tau_1),\tau_2,w_p(\tau_2)\mid^2)&=G_N(\tau_1,\tau_2)+G_N(w_q(\tau_1),w_p(\tau_2))
        \\&-G_N(w_q(\tau_1),\tau_2)-G_N(\tau_1,w_p(\tau_2)
        \end{split}
    \end{equation}
\end{proposition}
\begin{proof}
    The function 
    \begin{equation}
        f:z\rightarrow \log\Biggl(\frac{\mid j_{N}(z)-j_{N}(\tau_2)\mid^2}{\mid j_{N}(z)-j_{N}(w_p(\tau_2))\mid^2}\Biggl)
        -(G_N(z,\tau_2)-G_N(z,w_p(\tau_2)))
    \end{equation}
is harmonic and defined over $\mathbf{H}$ even when $z=\tau_2 \textup{ or } w_p(\tau_2) \mod R^1$ because both sides have the same logarithmic pole here. As $R^1\backslash\mathbf{H}$ is compact, this function is constant. In particular, we have
\[f(\tau_1)-f(w_q(\tau_1))=0\] 
which conclude the proof.
\end{proof}
Instead of directly relating $\mathcal{J}_N(D_1,D_2)$ to $c_1$, we will take advantage of Proposition~\ref{Greenf} and evaluate the Green's function $G_N$ at CM-points. To do so, we will rewrite the sum (\ref{eq green function}) over $\gamma\in R^1$ defining $G_N$, as a sum over $\nu\in \delta_F^{-1}\mathbf{p}\mathbf{q}$. Our main tool will be a quadratic form $\textup{det}_{F,\alpha}:B_{N}\rightarrow F$ defined in Proposition \ref{quadform}.
\subsection{Embeddings}
In this subsection, we describe the action of $\textup{Pic}(K_1)\times\textup{Pic}(K_2)$ on the embeddings and introduce the reflex ideal associated to a pair of embeddings. This will be useful to understand which elements $\nu\in F$ are represented by the quadratic form $\textup{det}_{F,\alpha}$ (see Proposition \ref{counting thm}).
\par
We will follow \cite{Voightquaternion} Chapter 30 and \cite{Daas} Section 3.4 for this part.
Let $i=1,2$ and 
\[\text{Emb}_R(\mathcal{O}_i)\coloneqq \{R^1-\text{conjugacy classes of  } \alpha:\mathcal{O}_i\xhookrightarrow{} R \},\] 
where $R^1$ acts by conjugation on the embeddings. Note that in our case, any embedding in $R$ is optimal because $R$ is the only maximal order up to conjugation.
\par
Let $\omega_p\in R$ be an element of norm $p$, let $\alpha_i\in\text{Emb}_R(\mathcal{O}_{i})$ and define $\theta_i$ as the morphism given by the following diagram
\begin{center}
\begin{tikzcd}
\mathcal{O}_{K_i} \arrow[hookrightarrow]{r}{\alpha_i}  \arrow{rd}{\theta_i} 
  & R \arrow[two heads]{d}{\pi_p}\\
    & R/\omega_p
\end{tikzcd}
\end{center}
with $\pi_p$ the projection on $R/\omega_p\simeq\mathbf{F}_{p^2}$. 
For a given $\alpha=(\alpha_1,\alpha_2)$ we have a well defined ring homomorphism 
\begin{equation}
\theta_p=\theta_1\otimes\theta_2:\mathcal{O}_L\rightarrow R/\omega_p.
\end{equation}
The restriction $\theta_p\mid_{\mathcal{O}_F}$ is not injective ($p$ is clearly in the kernel), so the kernel is an ideal of norm $p$ (there is no prime ideal of norm $p^2$ in $\mathcal{O}_F$ because $p$ splits in $\mathcal{O}_F$). The kernel must be $\mathbf{p}_1$ or $\mathbf{p}_2$. 
\begin{definition}
    The kernel of $\theta_p\mid_{\mathcal{O}_F}$ is called the reflex prime at $p$. We do the same construction for the reflex prime at $q$. We call the reflex ideal of the pair of embeddings $\alpha$ the product of its reflex prime at $p$ and at $q$.
\end{definition}
    By the commutativity of the quotient $R/\omega_p$, our definition of the reflex prime is independent of both the choice of $\omega_p$ and the representative of the class of $\alpha_i$ modulo $R^1$.
\begin{remark}
    We could also define an orientation at $p$ for each $\alpha_i$ as in \cite{BertoliniDarmon1996} section 2.2. To go from their orientation to our reflex prime, we choose a sign for the orientation, and we say that the reflex prime at $p$ is $\mathbf{p}_1$ when the orientations of the $\alpha_i$ at $p$ are the same, and $\mathbf{p}_2$ otherwise.
\end{remark}
Let $I$ be an ideal of $\mathcal{O}_{1}$ and $\alpha_1\in\text{Emb}_R(\mathcal{O}_{1})$, then $\alpha_1(I)R$ is an ideal of $R$. As $R$ is principal, we can write $\alpha_1(I)R=\zeta R$ for some $\zeta\in R$. We define the action of $I$ on $\alpha_1$ as 
\begin{equation}\label{action of Pic}
  I\mid \alpha_1\coloneqq \zeta^{-1}\alpha_1\zeta.   
\end{equation}
    Let $J=\mu I$ with $\mu\in K_1$, the commutativity of $K_1$ implies $J\mid\alpha_1=I\mid\alpha_1$. We can then define this action for $[I]\in\textup{Pic}(K_1)$.
\par
 We make $\alpha_2(I)$ acts from the right on $R$ to define the action of $\textup{Pic}(K_2)$. Let $\sigma_1,\sigma_2\in\textup{Pic}(K_1)\times\textup{Pic}(K_2)$, we will write the action of $\sigma_1,\sigma_2$ on $\alpha=(\alpha_1,\alpha_2)$ as $\alpha [\sigma_1,\sigma_2]\coloneqq (\sigma_1\mid\alpha_1,\sigma_2\mid\alpha_2)$. 
The following proposition explains how this action interacts with the reflex primes. A proof can be found in \cite{Daas} Lemma 3.4.1.
\begin{lemma}
    The action of $\textup{Pic}(K_1)\times\textup{Pic}(K_2)$ on the embeddings does not change the reflex primes at $p$ and $q$.  
\end{lemma}
Keeping in mind that the action of  $W_N\times\textup{Pic}(K_i)$ on the embeddings of $K_i$ is free and transitive, we now state how the Atkin--Lehner group $W_N$ interacts with the reflex ideal. 
Recall the definition of the Atkin--Lehner group, 
\begin{equation}\label{atkin lehner def}
    W_N\coloneqq \frac{\textup{Normaliser}_{B_N^*}(R^1)}{\mathbb{Q}^\times R^1}=\{w_k:k|N\}\cong \prod_{k|N}\mathbb{Z}/2\mathbb{Z}.
\end{equation}
The last two equalities are given early in \cite{Voightquaternion} chapter 43.
Let $p|N$, then $w_p$ acts on an embedding $\alpha_i$ of $K_i$ by conjugation by an element of $B_N$ of norm $p$. For any $k_1,k_2|N$ and $\alpha=(\alpha_1,\alpha_2)$, we write 
\[\alpha[w_{k_1},w_{k_2}]\coloneqq (w_{k_1}\vert\alpha_1,w_{k_2}\vert\alpha_2).\]
 We give Lemma 3.4.2 from \cite{Daas}.
\begin{lemma}\label{atkin lehner reflex prime}
    Let $\alpha=(\alpha_1,\alpha_2)$ with $\mathbf{p}\mathbf{q}$ as reflex ideal, then $\alpha[1,w_p]$ and $\alpha[w_p,1]$ have reflex ideal $\mathbf{p}'\mathbf{q}$. The pairs of embeddings $\alpha[1,w_q]$ and $\alpha[w_q,1]$ have reflex ideal $\mathbf{p}\mathbf{q}'$.
\end{lemma}
In the next part, we define an $F$-quadratic form $\textup{det}_{F,\alpha}$. The reflex prime and the action of $\textup{Pic}(K_1)\times\textup{Pic}(K_2)$ play a crucial role in the way we count representation numbers of this form.

\subsection{\texorpdfstring{An $F$-quadratic form}{An F-quadratic form}}\label{F quad form}
In this subsection, we construct a quadratic form $\textup{det}_{F,\alpha}:B_{N}\rightarrow F$ for each pair of embeddings $\alpha=(\alpha_1,\alpha_2)$. We use it in the next subsection to rewrite the series in $\gamma\in R^1$ appearing in Equation (\ref{eq green function}) as a series in $\nu\in F$ to relate it to $c_1$. To do so, we will need the explicit expression of $\textup{det}_{F,\alpha}$ given by Proposition \ref{quadform} and a way to count how many times each $\nu\in F$ is represented by this quadratic form. This will be the object of Proposition \ref{counting thm}. The results of this subsection are an analog to indefinite quaternion algebra of the results of \cite{Daas} section 4. We merely state them, as the proofs in the indefinite case are almost identical.
\par
Let $\text{Nm}:B_{N}\rightarrow \mathbb{Q}$ be the quaternionic norm, an indefinite quadratic form. We will refine this form to an $F$-quadratic form by taking the trace of $F/\mathbb{Q}$. To do so, let's fix $\alpha_i$ an embdedding of $\mathcal{O}_i$ as before.
\par
The pair $\alpha=(\alpha_1,\alpha_2)$ turns $B_{N}$ into a 1-dimensional $L$-vector space  ($L$ and $B_N$ are both $\mathbb{Q}$-vector space of dimension 4). Let $x_i\in K_i$, then the action of $(x_1,x_2)$ on $\gamma\in B_{N}$ is defined by $(x_1,x_2)\ast \gamma=\alpha_1(x_1)\gamma\alpha_2(x_2)$. Because $L=K_1K_2$, we extend to $L$ by $\mathbb{Q}$-linearity. We give the archimedean equivalent of Theorem 4.4.9 of \cite{Daas}.
\begin{proposition}\label{quadform}
    There is a unique $F$-quadratic form $\textup{det}_{F,\alpha}:B_{N}\rightarrow F$ such that for all $\gamma\in B_{N}$, 
    \[\textup{Tr}_{F/\mathbb{Q}}(\textup{det}_{F,\alpha}(\gamma))=\textup{Nm}(\gamma).\]
Furthermore, if $\alpha_i(\sqrt{D_i})=\begin{pmatrix} b_i & 2c_i \\ 2a_i & -b_i
\end{pmatrix}$ and $\tau_i=\frac{-b_i+\sqrt{D_i}}{2a_i}$, then
\begin{equation}
    \textup{det}_{F,\alpha}(\gamma)=-\textup{Nm}(\gamma)\frac{\mid \tau_1-\gamma\tau_2\mid^2}{4\text{\normalfont{Im}}(\tau_1)\text{\normalfont{Im}}(\gamma\tau_2) }.
\end{equation}
\end{proposition}
Let $\overline{\alpha_i}$ be the composition of $\alpha_i$ and the non trivial automorphism of $K_i$. It follows from Corollary 4.4.5 in \cite{Daas} that if $\alpha^*=(\overline{\alpha_1},\alpha_2)$, then $\textup{det}_{F,\alpha}(\gamma)$ and $\textup{det}_{F,\alpha^*}(\gamma)$ are galois conjugate in $F$. Because of the symmetry in $\nu$ and $\nu'$ in $c_1$, the choice of square root $\sqrt{D_i}$ is not important in this proposition.
Furthermore, taking $\overline{\alpha_1}$ will swap the role of $\tau_1$ and $\overline{\tau_1}$. This gives us the archimedian equivalent of Lemma 6.1.2.
\begin{lemma}
    With the same notation as before, let $\gamma\in B_N$, then 
    \[ \frac{\textup{det}_{F,\alpha}(\gamma)}{\textup{det}_{F,\alpha}(\gamma)'}=-\frac{\mid\tau_1-\gamma\tau_2\mid^2}{\mid\tau_1-\gamma\overline{\tau_2}\mid^2}.\]
\end{lemma}

\begin{remark}\label{remark sign det}
    A direct consequence of this lemma is that $\textup{det}_{F,\alpha}(\gamma)$ is of negative norm. When we work in a definite quaternion algebra like in \cite{Daas}, this element is always totally positive (Proposition 4.4.3).  Looking at $c_1$ from Corollary \ref{cor}, we are indeed taking the sum over $\nu$ of negative norm. Now taking $\nu=\textup{det}_{F,\alpha}(\gamma)$ and counting how many times each $\nu$ appears seems a reasonable strategy.
\end{remark}
We now need to know how many times each $\nu \in \delta_F^{-1}\mathbf{p}\mathbf{q} $ appears, this next proposition is a consequence of the discussions in \cite{HowardYang}. An explicit bijection is also given by Theorem 4.0.2 in \cite{Daas}. The proof is given for a definite quaternion algebra but we can adapt it to our setting. 

\begin{proposition}
\label{counting thm}
    Let $\alpha=(\alpha_1,\alpha_2)$ and $\mathbf{p}\mathbf{q}$ the reflex ideal associated to $\alpha$. Let $\nu\in F$ of negative norm, then,
    \begin{align*}
    \#\{(\gamma,\sigma_1,\sigma_2)\in \mathcal{O}_1^\times\backslash R/ \mathcal{O}_2^\times\times \textup{Pic}(K_1)\times\textup{Pic}(K_2),\,\, \textup{det}_{F,\alpha[\sigma_1,\sigma_2]}(\gamma)=\nu\}
    =\sigma_{0,\chi}((\delta_F^{-1}\mathbf{p}\mathbf{q})^{-1}(\nu)).
        \end{align*}
\end{proposition}
\subsection{End of the proof}
We start this Subsection by proving the following result stated in the Introduction.
\begin{proposition}\label{rewriting J_n}
    \begin{equation}
\mathcal{J}_N(D_1,D_2)
=\Biggl(\prod_{\substack{(\sigma_1,\sigma_2)\in G}}
\biggl|J_N\Bigl(P_1^{\sigma_1},w_q(P_1)^{\sigma_1},P_2^{\sigma_2},w_p(P_2)^{\sigma_2}\Bigl)\biggl|^2\Biggl)^{\frac{4}{w_1w_2}},
\end{equation}
where $G=\textup{Pic}(K_1)\times\textup{Pic}(K_2). $
\end{proposition}
\begin{proof}
    Following \cite{Daas} Lemma 3.4.4, we have $P_i'\in \textup{Pic}(K_i)\cdot w_q(P_i)$. The Atkin--Lehner involution $w_N$ acts by Galois action on $j_N(P_i)$. Modulo the action of $ \textup{Pic}(K_i)$, it must act by complex conjugation. Since $J_N(P_1,w_q(P_1),P_2,w_p(P_2))$ only lives in a field of absolute degree $2h_1h_2$, we can rewrite $\mathcal{J}_N(D_1,D_2)$ as
\begin{equation}   
\mathcal{J}_N(D_1,D_2)=\textup{Nm}_L^{H_1H_2}\bigl(\bigl|J_N(P_1,w_q(P_1),P_2,w_p(P_2))\bigl|^2\bigl)^{\frac{4}{w_1w_2}}.
\end{equation}
Taking the norm over $L$ corresponds to taking the product over the action of $\textup{Pic}(K_1)\times\textup{Pic}(K_2)$. As a consequence of Lemma $2.5$ of \cite{BertoliniDarmon1996}, for any $\sigma_i\in\textup{Pic}(K_i)$, we have 
\[w_N\bigl(w_N(P_i)^{\sigma_i}\bigl)\in\textup{Pic}(K_i)\cdot P_i.\]
By taking the product on $G$, we can conclude.
\end{proof}
We already know from Corollary \ref{cor} and Proposition \ref{prop b1} that
\begin{equation}
c_1=b_1=-\log\Biggl(\prod_{\substack{x^2<D\\x^2\equiv D \bmod{4N}}}\mathbf{F}\biggl(\frac{D-x^2}{4N}\biggl)^{\delta(x)}\Biggl).
\end{equation}
We end the proof of Theorem \ref{padicgz} by proving the following proposition.
\begin{proposition}\label{valeur c1}
\begin{equation}
c_1=-\frac{4}{w_1w_2}\log\Biggl(\prod_{\substack{(\sigma_1,\sigma_2)\in G}}
\biggl|J_N\Bigl(P_1^{\sigma_1},w_q(P_1)^{\sigma_1},P_2^{\sigma_2},w_p(P_2)^{\sigma_2}\Bigl)\biggl|^2\Biggl).
\end{equation}
\end{proposition}
\begin{proof}

Let $\alpha=(\alpha_1,\alpha_2)$ an embedding with reflex ideal $\mathbf{pq}$. Let $\tau_i\in\CC$ taken
as in Proposition \ref{quadform}. By the same Proposition, we know that
\[\frac{\mid \tau_1-\gamma\tau_2\mid^2}{2\text{\normalfont{Im}}(\tau_1)\text{\normalfont{Im}}(\gamma\tau_2)} =-2\textup{det}_{F,\alpha}(\gamma).
\]
By changing $\tau_i$ to $\overline{\tau_i}$, we have $-2\textup{det}_{F,\alpha}(\gamma)'$ instead. We can assume without loss of generality that both $\tau_1,\tau_2$ are in $\mathbf{H}$. We evaluate the Green's function $G_N$ given by Equation (\ref{eq green function}) at $\tau_1,\tau_2$.
 When Re$(s)>1$, using the invariance of $\textup{det}_{F,\alpha}$ by $\mathcal{O}_i^\times$, we have
\begin{equation}
\begin{split}
G_{N,s}(\tau_1,\tau_2)=-\frac{\omega_1\omega_2}{2}\sum_{\gamma\in \mathcal{O}_1^\times\backslash R^1/ \mathcal{O}_2^\times}Q_{s-1}(1-2 \textup{det}_{F,\alpha}(\gamma)).
\end{split}
\end{equation}
We know by Proposition \ref{counting thm} that for $\gamma\in R^1$, 
\begin{equation}
    \nu\coloneqq \textup{det}_{F,\alpha}(\gamma)\in\delta_F^{-1}\mathbf{p}\mathbf{q}\quad\textup{and}\quad \textup{Tr}_{F/\mathbb{Q}}(\nu)=1.
\end{equation}
 Furthermore, $\nu<0$ by the explicit expression of $\textup{det}_{F,\alpha}$ and $\nu'>0$ by Remark \ref{remark sign det}. We now sum over $\nu$ and count how many times each $\nu$ appears in the sum. Instead of taking only one embedding, we will sum on $\alpha[\sigma_1,\sigma_2]$ for 
$\sigma_i\in\textup{Pic}(K_i)$, $i=1,2$. We write $G_{N,s}(\alpha)\coloneqq G_{N,s}(\tau_1,\tau_2)$ with $\tau_i\in\mathbf{H}$ the fixed point of $\alpha_i$. Using Proposition \ref{counting thm} to count how many times each $\nu$ appears, we get
\begin{equation}
\sum_{(\sigma_1,\sigma_2)\in G}G_{N,s}(\alpha[\sigma_1,\sigma_2])
= -\frac{\omega_1\omega_2}{2}\sum_{\substack{\nu\in \delta_F^{-1}\mathbf{p}\mathbf{q}\\ \nu'>0>\nu\\ \textup{Tr}_{F/\mathbb{Q}}(\nu)=1}}\sigma_{0,\chi}((\delta_F^{-1}\mathbf{p}\mathbf{q})^{-1}(\nu))Q_{s-1}(1-2\nu).
\end{equation}
Instead of taking only $G_{N,s}(\alpha)$, let's do the same computation for 
\begin{equation}\label{eq define mathcal G}
\mathcal{G}_{N,s}(\alpha)\coloneqq G_{N,s}(\alpha)+G_{N,s}(\alpha[w_q,w_p])-G_{N,s}(\alpha[w_q,1])-G_{N,s}(\alpha[1,w_p]).
\end{equation}
By Lemma \ref{atkin lehner reflex prime}, $\alpha[w_q,w_p]$, $\alpha[1,w_p]$, $\alpha[w_q,1]$  have, respectively, $\mathbf{p}'\mathbf{q}'$, $\mathbf{p}'\mathbf{q}$, $\mathbf{p}\mathbf{q}'$ for reflex ideal. 
Taking $\mathbf{p}\mathbf{q}=\mathbf{p}_1\mathbf{q}_1$ and summing on the action of $\textup{Pic}(K_1)\times\textup{Pic}(K_2)$ gives us
\begin{align*}
\lim_{s\rightarrow1}\sum_{(\sigma_1,\sigma_2)\in G}\mathcal{G}_{N,s}(\alpha)^{(\sigma_1,\sigma_2)}=-\frac{w_1w_2}{4}c_1.
\end{align*}
Where $\sigma_i$ acts on the argument of $G_{N,s}$ in Equation (\ref{eq define mathcal G}). We finally call Proposition \ref{Greenf} relating the values of $\mathcal{G}_N$ and $J_N$
\[c_1=-\frac{4}{w_1w_2}\log\Biggl(\prod_{\substack{(\sigma_1,\sigma_2)\in G}}
\biggl|J_N\Bigl(P_1^{\sigma_1},w_q(P_1)^{\sigma_1},P_2^{\sigma_2},w_p(P_2)^{\sigma_2}\Bigl)\biggl|^2\Biggl)=-\log\Bigl(\mathcal{J}_N(D_1,D_2)\Bigl).\]
 This concludes the proof of Proposition \ref{valeur c1}. 
\end{proof}
We can finally end the proof of Theorem~\ref{padicgz} by invoking Corollary~\ref{cor} and Proposition~\ref{prop b1}:
\[-\log\Bigl(\mathcal{J}_N(D_1,D_2)\Bigl)=c_1=b_1=-\log\Biggl(\prod_{\substack{x^2<D\\x^2\equiv D \bmod{4N}}}\mathbf{F}\biggl(\frac{D-x^2}{4N}\biggl)^{\delta(x)}\Biggl).\]
As the arguments on each side that are positive reals, we can exponentiate to conclude.
\qed
\begin{remark}
    Choosing the reflex primes $\mathbf{p}$ and $\mathbf{q}$ amounts to exactly choosing a square root of $D\mod 2N$. Choosing $\delta(x)$ such that our choice of root matches the one given by the reflex primes appearing here lets us lift the power ambiguity $\pm1$ appearing in Theorem \ref{padicgz}.
\end{remark}
\section{\texorpdfstring{Interpretation of the higher coefficients $a_m$}{Interpretation of the higher coefficients aₙ}}
So far, we only used the coefficient $a_1$ from Theorem \ref{thm mod}. We end this paper by making sense of the general coefficient $a_m$ in two cases. First, we use them to give a factorization of the cross-ratio with the action of the Hecke operators (see Theorem \ref{Hecke}). Then, for general $N$, we show that $a_m$ is the height pairing of CM points on the Shimura curves (Theorem \ref{f heights}).
\subsection{\texorpdfstring{When $N=6,10,22$}{When N=6,10,22}}
For $m$ prime to $N=6,10,22$, we define the action of the Hecke operators on $\mathcal{J}_N$. Let $R^m\subset R $ be the elements of norm $m$ in $R$,
\[\mathcal{J}_N^m(D_1,D_2)\coloneqq \prod_{\gamma\in R^1 \backslash R^m}\textup{Nm}_\mathbb{Q}^{H_1H_2}\Bigl(J_N(P_1,w_q(P_1),\gamma P_2,\gamma w_p(P_2))\Bigl)^{\frac{2}{w_1w_2}}.\]
We have the following factorization.
\begin{theorem}\label{Hecke}
    Let $N=6,10,22$ and $m$ prime to $N$, with the same notation as Theorem \ref{padicgz},
\begin{equation}
\mathcal{J}_N^m(D_1,D_2)^{\pm1}=\prod_{\substack{x^2<m^2D\\x^2\equiv m^2D \bmod{4N}}}\mathbf{F}\biggl(\frac{m^2D-x^2}{4N}\biggl)^{\delta_m(x)}.
\end{equation}
Where $\delta_m(x)\coloneqq \delta(\frac{x}{m})$ for $x^2\equiv m^2D [4N]$.
\end{theorem}
\begin{proof}
Let $\alpha=(\alpha_1,\alpha_2)$ a pair of embeddings with reflex ideal $\mathbf{pq}$. Let $\tau_i\in \mathbf{H}$ be the fixed point of $\alpha_i$. 
   We already proved the case $m=1$. Let $m$ be prime to $N,$ by Corollary \ref{cor}, we know that $b_m=c_m$. Using Proposition 1.6.1 from \cite{Daas}, we know that $b_m$ equals the right side of the equality.
  To deal with $c_m$, we write for Re$(s)>1$,
   \begin{equation}
   \begin{split}
   \sum_{\gamma_1\in R^1 \backslash R^m}G_{N,s}(\tau_1,\gamma_1\tau_2)=\sum_{\gamma\in R^m/\{\pm1\}}-2Q_{s-1}\Biggl(1+\frac{\mid \tau_1-\gamma\tau_2\mid^2}{2\text{\normalfont{Im}}(\tau_1)\text{\normalfont{Im}}(\gamma\tau_2)}\Biggl).
    \end{split}
   \end{equation}
  We are now exactly in the case of the proof of Proposition \ref{valeur c1}. The fact that the sum is taken over $\gamma\in R^m$ means that 
  \begin{equation}
\nu\coloneqq \textup{det}_{F,\alpha}(\gamma)=-m\frac{\mid \tau_1-\gamma\tau_2\mid^2}{4\text{\normalfont{Im}}(\tau_1)\text{\normalfont{Im}}(\gamma\tau_2)}
\end{equation}
  is of trace $m$. Following the steps of the proof of Proposition \ref{valeur c1}, we sum over the action of $\textup{Pic}(K_1)\times\textup{Pic}(K_2)$ and apply Proposition \ref{counting thm} to prove that \begin{equation}
c_m=\mathcal{J}_N^m(D_1,D_2).
\end{equation}
\end{proof}

\begin{remark} 
 As in \cite{GrossKohnenZagier87} \S3.2, Proposition~1, one can construct a modular form 
$ g \in S_k(\Gamma_0(N))$ of higher weight in the same manner as \( f \). 
The argument in the proof of Theorem~\ref{Hecke} shows that a sum of values of the higher Green's function $G_{N,k}$ evaluated at CM points on the Shimura curve  $X_N$  appears in the Fourier coefficients of $g$. The modular curve case of this statement is given by \cite{GrossKohnenZagier87} \S5, Theorem 2. 
Applying a suitable Hecke operator that annihilates $S_k(\Gamma_0(N))$, we deduce a rationality result for this sum. 
This is a special case corresponding to averaging over the class group of the more general results of \cite{Li_higher_green_function}.

\end{remark}

\subsection{\texorpdfstring{For general $N$}{For general N}}
We end this paper by relating the Fourier coefficients of $f$ from Theorem \ref{thm mod} with the height pairing of some points on the Shimura curve. We first recall the case of the modular curve done in \cite{GrossKohnenZagier87}.
\par
Let $X_0(N)=\Gamma_0(N)\backslash\overline{\mathbf{H}}$ be the modular curve of level $N$ and let $w_N$ be the Atkin Lehner involution. Let
$\langle\,\cdot \, , \cdot \, \rangle_{X_0(N)}$ be the Néron's height pairing on the Jacobian of $X_0(N)/w_N$.
We define \[P_D^*\coloneqq P_{D,a}^*-P_{D,b}^*\]
following the notation of \cite{GrossKohnenZagier87} $\S4.1$.
From the same section, we know that the function $f$ of Theorem \ref{thm mod} can be written, up to an oldform, as
\begin{equation}\label{GKZ heights}
    f=\sum_{m>0}\langle P_{D_1},T_m(P_{D_2})\rangle_{X_0(N)} q^m.
\end{equation}
Note that the results of \cite{GrossKohnenZagier87} are under the hypothesis that $p$ and $q$ are split in $K_i$. 
\par
In the case of the Shimura curve (thus, with $p,q$ inert in $K_i$), the result will be very similar. Let $\langle\,\cdot \, , \cdot \, \rangle_N$ be the Néron's height pairing for the Shimura curve $X_N(\CC)$. We decompose $\langle\,\cdot \, , \cdot \, \rangle_N$ into its archimedean and non archimedean part
\[\langle\,\cdot \, , \cdot \, \rangle_N=\langle\,\cdot \, , \cdot \, \rangle_{N,\infty}+\sum_{p}\langle\,\cdot \, , \cdot \, \rangle_{N,p}.\]
Let's define the divisors we are interested in.
Let $K_1,K_2$ define as before and $P_i$ the CM point associated to an embedding of $K_i$. Define
 \[P_1\coloneqq \frac{2}{w_1}\sum_{\sigma_1\in\textup{Pic}(K_1)}(P_1-w_q(P_1))^{\sigma_1}\]
 and
\[P_2\coloneqq \frac{2}{w_2}\sum_{\sigma_2\in\textup{Pic}(K_2)}(P_2-w_p(P_2))^{\sigma_2}.\]
\begin{proposition}  For $m $ prime to $N$, let $T_m$ be the $m$-th Hecke operator, then
\begin{equation}
c_m=-\langle P_1,T_mP_2\rangle _{N,\infty}.
\end{equation}
\end{proposition}
\begin{proof}
Let $G_N$ be the Green's function as defined in Subsection \ref{subsection green function}. There exists a $Q\in X_N(\CC)$ such that for all  $P_1\neq P_2$ different from $Q$, 
    \[G_N(P_1,P_2)=\langle P_1-Q,P_2-Q\rangle _\infty.\]
    By bi-linearity and by adding and subtracting $Q$, we have
   \begin{equation}
   \begin{split}
        \langle P_1-w_q(P_1),P_2-w_p(P_2)\rangle _{N,\infty}&
    =G_N(P_1,P_2)+G_N(w_q(P_1),w_p(P_2))
    \\&-G_N(w_q(P_1),P_2)-G_N(P_1,w_p(P_2)).
         \end{split}
        \end{equation}
We sum on $\textup{Pic}(K_i)$ and invoke Proposition \ref{valeur c1}
to conclude in the case $m=1$. The general case is done in the same way using the linearity of $T_m.$
\end{proof}
For finite places, we invoke the main Theorem of \cite{anphil} and \cite{Daas} Theorem $3.2.5$. 
\begin{proposition} Let $m$ be prime to $N$,
\begin{equation}
b_m=-\sum_{\substack{x^2<m^2D\\x^2\equiv m^2D [4N]}}\log \biggl(\mathbf{F}\biggl(\frac{m^2D-x^2}{4N}\biggl)^{\delta_m(x)}\biggl)=\sum_p\langle P_1,T_mP_2\rangle _{N,p}.
\end{equation}
\end{proposition}
Putting everything together, we have the following theorem.
\begin{theorem}\label{f heights}
    Let $f=\sum_m a_m\in S_2(\Gamma_0(N))$ be the modular form constructed in Theorem 
    \ref{thm mod}, then for $m>0$ prime to $N$
    \begin{equation}
        a_m=\langle P_1,T_mP_2\rangle_N.
    \end{equation}
\end{theorem}
\begin{remark}
  This modular form $f$, used in our proof and in \cite{GrossKohnenZagier87} sees both the heights on the modular curve $X_0(N)$ and the Shimura curve $X_N$ but never at the same time : in the modular case, the primes $p$ and $q$ are split in $K_i$, whereas in the Shimura curve case they are inert.
\end{remark}
\begin{remark}
    The modular form constructed in \cite{Daas} by ordinary projection (Proposition 6.3.1) should hold a similar interpretation in terms of $p$-adic Schneider height pairing, see Remark 2.1.3.
\end{remark}

\newpage
\bibliographystyle{alpha}
\bibliography{ref}
\end{document}